%% file: ms.tex
\documentclass[a4paper,10pt]{article}
\usepackage{mypackages}
\usepackage{mymacros}
\usepackage{mystyle}

\title{Coarse compactifications of proper metric spaces}
\author{Elisa Hartmann\thanks{Department of Mathematics, Karlsruhe Institute of 
Technology}}

\begin{document}

\maketitle

\begin{abstract}

This paper studies coarse compactifications and their boundary.

 We introduce two alternative descriptions to Roe's original definition of coarse compactification. One approach uses bounded functions on $X$ that can be extended to the boundary. They satisfy the Higson property exactly when the compactification is coarse. The other approach defines a relation on subsets of $X$ which tells when two subsets closure meet on the boundary. A set of axioms characterizes when this relation defines a coarse compactification. Such a relation is called large-scale proximity.

Based on this foundational work we study examples for coarse compactifications Higson compactification, Freudenthal compactification and Gromov compactification. For each example we characterize the bounded functions which can be extended to the coarse compactification and the corresponding large-scale proximity relation.

We provide an alternative proof for the property that the Higson compactification is universal among coarse compactifications. Furthermore the Freudenthal compactification is universal among coarse compactifications with totally disconnected boundary. If $X$ is hyperbolic geodesic proper then there is a closed embedding $\nu(\R_+)\times \partial X\to \nu(X)$. Its image is a retract of $\nu(X)$ if $X$ is a tree.
\end{abstract}

\tableofcontents

\input{Introduction}
\input{The_coarse_Category}
\input{Original}

\input{Relations}

\input{Bounded_Functions}

\input{Functoriality}
\input{Higson}

\input{Freudenthal}
\input{Gromov}
\input{Remarks}

\bibliographystyle{halpha-abbrv}
\bibliography{mybib}

\address

\end{document}

%% file: Introduction.tex
\section{Introduction}

This paper studies a class of compactifications of proper metric spaces which contains the Higson compactification, Gromov compactification and a coarse version of the Freudenthal compactification. With every such space $X$ one can associate a coarse structure\footnote{See Definition~\ref{defn:coarsestructure} which defines a coarse structure given a metric space}. In \cite{Roe2003} Roe introduced a class of compactifications which are compatible with this structure. We provide an equivalent definition:

\begin{defn}
Let $X$ be a proper metric space and $\bar X$ a compactification of $X$. Then $\bar X$ is a \emph{coarse} compactification if for every two nets $(x_i)_i,(y_i)_i\s X$, such that $(x_i,y_i)_i$ is an entourage in $X$, both nets have the same limit points on the boundary. 
\end{defn}

Instances of coarse compactifications have been studied by many authors, see e.g. \cite{Mine2015,Keesling1994,Protasov2019,Foertsch2003,Protasov2011,Protasov2015,Protasov2005,Benakli2002,Kalantari2016,Kalantari2015,Cornulier2019}. In this paper we give two new descriptions of coarse compactifications which are equivalent to the original one. One of the two equivalent definition starts with a relation on subsets of $X$, the other definition uses bounded functions on $X$ which can be extended to the compactification.

Every coarse compactification $\bar X$ of a proper metric space $X$ gives rise to a relation $r_{\bar X}$ on subsets of $X$ which tells when the closure of two sets meet on the boundary. Specifically denote by $\partial X$ the boundary $\bar X\setminus X$. If $A,B\s X$ are subsets then 
\[
A r_{\bar X} B \Leftrightarrow (\bar A\cap \partial X)\cap (\bar B\cap\partial X)\not=\emptyset.
\]
A set of axioms tells if a relation on subsets of $X$ comes from a coarse compactification. Such a relation is then called large-scale proximity in the sense of the following definition. 
\begin{defn}
If $X$ is a proper metric space a relation $r$ on subsets of $X$ is called \emph{large-scale proximity} if
\begin{enumerate}
\item if $B\s X$ then $B\bar r B$ if and only if $B$ is bounded;
\item  $A r B$ implies $B r A$ for every $A,B\s X$;
\item if $A,A'\s X$ are subsets and $E\s X^2$ is an entourage with $E[A]\z 
A',E[A']\z A$ then $A r B$ implies $A' r B$ for every $B\s X$; 
\item if $A,B,C\s X$ then $(A\cup B)r C$ if and only if ($A r C$ or $B r C$);
\item if $A,B\s X$ with $A\bar r B$ then there exist $C,D\s X$  with $C\cup D=X$ 
and $C\bar r A,D\bar r B$.
\end{enumerate}
\end{defn}

Conversely given a large-scale proximity relation $r$ we can construct a coarse compactification $\bar X^r$ which induces this relation on subsets of $X$. This is done in Definitions~\ref{defn:relationtocompactification1},\ref{defn:relationtocompactification2}.

An entirely different approach characterizes coarse compactifications via the $C^*$-algebra $C_r(X)$ of bounded continuous functions on $X$ that can be extended to the boundary of the compactification. Every set of bounded continuous functions $\mathcal A$ on $X$ generates the smallest compactification $\bar X^{\mathcal A}$ such that functions in $\mathcal A$ can be extended to the boundary. More specifically we introduce a property on bounded continuous functions: 

\begin{defn}
A bounded continuous function $\varphi:X\to \R$ is called \emph{Higson} if for every entourage $E\s X^2$ the map 
\f{
d\varphi|_E:E &\to \R\\
(x,y)&\mapsto \varphi(x)-\varphi(y)
}
vanishes at infinity.
\end{defn}
Given a compactification $\bar X$ denote by $C_{\bar X}(X)$ the algebra of bounded continuous functions on $X$ that can be extended to $\bar X$. They must be Higson if $\bar X$ is coarse. Conversely if the algebra of bounded functions on $X$ that can be extended to the boundary of the compactification consists of Higson functions then $\bar X$ is coarse. 

We summarize those results in the following theorem:

\begin{thma}
\label{thm:a}
 If $X$ is a proper metric space and $\bar X$ a compactification of $X$ the following statements are equivalent:
 \begin{itemize}
  \item The compactification $\bar X$ is coarse.
  \item The relation $r_{\bar X}$ on subsets of $X$ is a large-scale proximity relation. 
  \item Every function in $C_{\bar X}(X)$ is Higson.
 \end{itemize}
 Moreover given a large-scale proximity relation $r$ the compactification $\bar X^r$ is coarse. If all functions in an algebra of bounded functions $\mathcal A$ on $X$ are Higson they generate a coarse compactification $\bar X^{\mathcal A}$. 
\end{thma}

In Theorem~\ref{thm:roevsclose} we translate Roe's original definition of coarse compactification to a definition which is more suitable for us. The equivalence of statements 1,2 in Theorem~\ref{thm:a} is shown in Theorem~\ref{thm:coarsetorelation}. The equivalence of statements 1,3 in Theorem~\ref{thm:a} is shown in Theorem~\ref{thm:coarsefunctions}.

We investigate three specific examples: The Higson compactification, the Freudenthal compactification and the Gromov compactification.

\begin{ex}The Higson compactification $hX=\nu(X)\cup X$ of a proper metric space $X$ is characterized by the following large-scale proximity relation: Two subsets $A,B\s X$ are called \emph{close}, written $A\close B$, if there exists an unbounded sequence $(a_i,b_i)_i\s A\times B$ and some $R\ge 0$ such that $d(a_i,b_i)\le R$ for every $i$. 

 Every Higson function on $X$ can be extended to the Higson corona $\nu(X)$.
\end{ex}

\begin{ex}
The Freudenthal compactification $\varepsilon X=\Omega X\cup X$ of a proper metric space $X$ is characterized by the following large-scale proximity relation $\close_f$: Two subsets $A,B\s X$  don't have a same end, written $A\notclose_f B$, if there exist $A'\z A,B'\z B$ with $A'\cup B'=X$ and $A'\notclose B'$. 

Let $x_0\in X$ be a basepoint. A bounded continuous map $\varphi:X\to \R$ is called \emph{Freudenthal} if for every $R\ge 0$ there exists $K\ge 0$ such that $d(x,y)\le R,d(x_0,x)\ge K, d(x_0,y)\ge K$ implies $\varphi(x)=\varphi(y)$. We write $C_f(X)$ for the ring of Freudenthal functions on $X$. Every bounded function that can be extended to the Freudenthal compactification is Freudenthal and every Freudenthal function can be extended to the boundary of the Freudenthal compactification.
\end{ex}

\begin{ex}
If a metric space $X$ is hyperbolic, proper the Gromov compactification $\bar X=\partial X\cup X$ is defined. The associated large-scale proximity relation $\close_g$ is defined by $A\close_g B$ if there are sequences $(a_i)_i\s A,(b_i)_i\s B$ such that 
\[
      \liminf_{i,j\to \infty}(a_i|b_j)=\infty.
\]

 If $X$ is hyperbolic a continuous function $\varphi:X\to \R$ is called \emph{Gromov} if for every $\varepsilon>0$ there exists $K>0$ such that
\[
      (x|y)>K \,\to\, |\varphi(x)-\varphi(y)|<\varepsilon.
\]
Every Gromov function can be extended to the Gromov boundary and every function that can be extended to the Gromov boundary is Gromov.
\end{ex}

\begin{rem}
We establish functoriality in the following way. In Proposition~\ref{prop:contravariant} we show the association of a coarse compactification is in a way contravariant on coarse maps. We can always pull back a coarse compactification along a coarse map. The reverse direction push out is not always possible. We can glue coarse compactifications along a coarse cover though which is done in Proposition~\ref{prop:ccsheaf}. Then Lemma~\ref{lem:ccsheaf} shows the poset of coarse compactifications is a sheaf on the Grothendieck topology of coarse covers on $X$.
\end{rem}

We now describe the results on the specific examples in detail. In particular the boundary of the Higson compactification retains information about the coarse structure since the Higson corona is a faithful functor \cite{Hartmann2019c}. This way not much information is lost if we restrict our attention to the boundary of a coarse compactification when studying coarse metric spaces. The Higson corona $\nu(X)$ of $X$ is connected if $X$ is one-ended. Aside from that and from being compact and Hausdorff the Higson corona does not have many nice property. The topology of the Higson corona does not have a countable base and is in fact is never metrizable \cite{Roe2003}. We provide a new proof that the Higson corona is universal among coarse compactifications. The original result can be found in \cite{Roe2003}.

\begin{thma}\name{Roe}
 If $X$ is a proper metric space the Higson compactification of $X$ is universal among coarse compactifications of $X$. This means a compactification of $X$ is a coarse compactification if and only if it is a quotient of the Higson compactification, where the quotient map restricts to the identity on $X$.
\end{thma}

This in particular implies that the boundary of every coarse compactification of $X$ is connected if $X$ is one-ended.

The space of ends of a topological space is well known and dates back to Freudenthal's works \cite{Freudenthal1931,Hopf1944, Freudenthal1945}. We construct a version of Freudenthal compactification on coarse proper metric spaces given both descriptions via a large-scale proximity relation and via bounded functions. The Proposition~\ref{prop:freudenthalequivcoarsetop} shows both the topological and the coarse version of Freudenthal compactifiaction agree on proper geodesic metric spaces. The space of ends gives information about the number of ends of a coarse metric space \cite{Hartmann2017a}. It is both metrizable and totally disconnected.

For the space of ends we can obtain a similar result as in the case of the Higson corona regarding universality.

\begin{thma}
If $X$ is a proper metric space the boundary of the Freudenthal compactification $\Omega X=\varepsilon X\ohne X$ of $X$ is totally disconnected. If $(\bar X,X)$ is another coarse compactification whose boundary is totally disconnected then it factors through $\varepsilon X$. This means there is a surjective map $\varepsilon X\to \bar X$ which is continuous on the boundary and the identity on $X$.
\end{thma}

The topology of the Gromov compactification is metrizable \cite{Benakli2002}. The usual description of its boundary is via geodesic rays or sequences that converge to infinity. We investigate in which way the Higson corona can be recovered from the Gromov boundary.

\begin{thma}
 Let $X$ be a hyperbolic geodesic proper metric space. Then there is a closed embedding $\Phi: \nu (\Z_+)\times \partial X\to \nu (X)$. The image of $\Phi$ is a retract if $X$ is a tree.
\end{thma}

All three instances of coarse compactification are functorial in that the boundary is a coarse invariant. The Higson corona and space of ends are even functors on coarse maps. This way we associate compact Hausdorff spaces to coarse proper metric spaces which serve in the classification of proper metric spaces according to their coarse geometry. This gives access to topological methods that can be used in the coarse setting.

%% file: The_coarse_Category.tex
\section{Notions in coarse geometry}
\label{sec:metric}

We consider metric spaces as coarse objects. The book \cite{Roe2003} introduces a more 
general notion of coarse spaces, which describes coarse structure in an abstract way. Since we only consider proper metric spaces as examples we do not need to do this here.

\begin{defn}
 A metric space $X$ is proper if the closure $\bar B$ in $X$ of every bounded subset $B\s X$ is compact.
\end{defn}

\begin{defn}
\label{defn:coarsestructure}
 Let $(X,d)$ be a metric space. Then the \emph{coarse structure associated to $d$} on $X$ consists of those subsets $E\s X\times X$ for which
 \[
  \sup_{(x,y)\in E}d(x,y)<\infty.
 \]
 We call an element of the coarse structure \emph{entourage}.  In what follows we assume the metric $d$ to be finite for every $(x,y)\in X\times X$.
\end{defn}

\begin{defn}
 A map $f:X\to Y$ between metric spaces is called 
 \begin{itemize}
  \item \emph{coarsely uniform} if $E\s X^2$ being an entourage implies that $\zzp f E$ is an entourage ;
  \item \emph{coarsely proper} if and if $A\s Y$ is bounded then $\iip f A$ is bounded.
  \item \emph{coarse} if it is both coarsely uniform and coarsely proper
 \end{itemize}
 Two maps $f,g:X\to Y$ between metric spaces are called \emph{close} if
 \[
 f\times g(\Delta_X)
 \]
 is an entourage in $Y$. Here $\Delta_X$ denotes the diagonal in $X\times X$.
\end{defn}

If $S\s X\times X,T\s X$ are subsets of a set we write
\[
 S[T]:=\{x:\exists y\in T,(x,y)\in S\}
\]
and $T^c=\{x\in X:x\not\in T\}$.

\begin{notat}
 A map $f:X\to Y$ between metric spaces is called
 \begin{itemize}
 \item \emph{coarsely surjective} if there is an entourage $E\s Y\times Y$ such that 
 \[
  E[\im f]=Y;
 \]
 \item \emph{coarsely injective} if for every entourage $F\s Y^2$ the set $\izp f F$ is an entourage in $X$.
\end{itemize}
  Two subsets $A,B\s X$ are called \emph{not coarsely disjoint} if there is an entourage $E\s X^2$ such that the set
  \[
  E[A]\cap E[B]
  \]
  is not bounded. We write $A\close B$ in this case.
\end{notat}

\begin{rem}
 We study metric spaces up to coarse equivalence. For a coarse map $f:X\to Y$ between metric spaces the following statements are equivalent:
 \begin{itemize}
  \item There is a coarse map $g:Y\to X$ such that $f\circ g$ is close to $id_Y$ and $g\circ f$ is close to $id_X$.
 \item The map $f$ is both coarsely injective and coarsely surjective.
 \end{itemize}
 We call $f$ a \emph{coarse equivalence} if one of the equivalent statements hold.
\end{rem}

\begin{notat}
If $X$ is a metric space and $U_1,\ldots, U_n\s X$ are subsets then $(U_i)_i$ are said to \emph{coarsely cover} $X$ if for every entourage $E\s X\times X$ the set 
\[
E[U_1^c]\cap \cdots \cap E[U_n^c]
\]
is bounded.
\end{notat}

%% file: Original.tex
\section{The original definition}
 
 In this chapter we introduce the class of compactifications which are coarse. 
 
 \begin{defn}
  A \emph{compactification} of a proper metric space (or more generally a locally compact Hausdorff topological space) is an open embedding $i:X\to \bar X$ such that $i(X)$ is dense in $\bar X$. We identify $X$ with the dense open set $i(X)\subset \bar X$.
 \end{defn}

Now we define when a compactification is coarse. The original definition was given in \cite[Theorem~2.27, Definition~2.28, Definition~2.38]{Roe2003}. We reproduce a slight modification of it.
\begin{defn}
Let $X$ be a proper metric space. If $\bar X$ is a compactification of $X$ with 
boundary $\partial X$ then the sets $E\s X\times X$ with
\[
      \bar E\cap( \partial X\times \bar X\cup\bar X\times \partial X)\s 
\Delta_{\partial X}
\]
define the \emph{topological coarse structure} associated to $\bar X$.

       A \emph{coarse compactification} of $X$ is a compactification whose 
topological coarse structure is finer than the originally given coarse 
structure on $X$. 
\end{defn}

Note in \cite[Definition~1.1]{Fukaya2018} a \emph{coarse compactification} of a 
proper metric space has been defined as a metrizable compactification $\bar X$ 
of $X$ equipped with a continuous map $f:hX\to \bar X$ which is the identity on 
$X$. This definition is different from our definition.

\begin{defn}
Let $X$ be a metric space. Two subsets $A,B\s X$ are called \emph{close} if there exists an unbounded sequence $(a_i,b_i)_i\s A\times B$ and some $R\ge 0$ such that $d(a_i,b_i)\le R$ for every $i$. We write $A\close B$ in this case. By \cite[Lemma~9, Proposition~10]{Hartmann2017b} the relation $\close$ is a large-scale proximity relation.
\end{defn}

\begin{thm}
\label{thm:roevsclose}
      Let $X$ be a proper metric space and $\bar X$ be a compactification of 
$X$. Then $\bar X$ is coarse if and only if for every two subsets $A,B\s X$ the 
relation  $A\close B$ implies $\bar A\cap \bar B\not=\emptyset$. 
\end{thm}
\begin{proof}
      Suppose $\bar X$ is coarse. Let $A,B\s X$ be two subsets with $A\close B$. Then there exist unbounded subsequences $(a_i)_i\s A,(b_i)_i\s B$ such that $(a_i,b_i)_i$ is an entourage. Then $\overline{(a_i,b_i)_i}\cap(\partial 
X\times\bar X\cup \bar X\times \partial X)\s \Delta_{\partial X}$. Thus if $p\in \partial X$ is a limit point of $(a_i)_i$ then it is also a limit point of $(b_i)_i$. Note $\overline{(a_i)_i}\cap (\partial X)\not=\emptyset$  since $(a_i)_i$ is unbounded.  This way we have shown $\overline{(a_i)_i}\cap\overline{(b_i)_i}\not=\emptyset$. This implies $\bar A\cap \bar B\not=\emptyset$.

Now suppose for every two subsets $A,B\s X$ the relation $A\close B$ implies 
$\bar A\cap \bar B\not=\emptyset$. Let $E\s X\times X$ be an entourage and 
$(x_i,y_i)_i\s E$ be a net such that $(x_i)_i\to p\in\partial X$ and 
$(y_i)_i\to q\in \bar X$. If $q\in X$ then there is an infinite subnet of $(y_i)_i$ contained in a ball around $q$. Then an infinite subnet of $(x_i)_i$ is contained in a (larger) ball around $q$, thus would have a limit point in this ball. This way we can conclude $q\in \partial X$.

If $(x_i)_i\cap (y_i)_i$ is bounded then remove those finitely many elements in 
the intersection and obtain $(x_{i_k})_k,(y_{i_k})_k$ subnets with the 
same limit points. Now $(x_{i_k})_k\close (y_{i_k})_k$ which implies 
$\overline{(x_{i_k})_k}\cap \overline{(y_{i_k})_k}\not=\emptyset$. This implies 
$p=q$.

If $(x_i)_i\cap (y_i)_i$ is not bounded then the subnet in the intersection 
converges to both $p$ and $q$. Thus $p=q$.

This way we have shown $ \bar E\cap( \partial X\times \bar X\cup \bar X\times 
\partial X)\s 
\Delta_{\partial X}$. Thus $\bar X$ is coarse.
\end{proof}

%% file: Relations.tex
\section{Large-scale proximity relations}

 In this chapter we study a relation on subsets of a proper metric space $X$ which induce the topology of a compactification $\bar X$. The relation is large-scale proximity as defined below exactly when the compactification $\bar X$ is coarse. Given a large-scale proximity relation $r$ on $X$ we are going to present two constructions of spaces $\partial_r X,\partial_r' X$ which happen to be boundaries of a coarse compactification $\bar X^r$ which induces the relation $r$ on subsets of $X$. 

 \begin{defn}
\label{defn:closerelation}
A relation $r$ on subsets of a metric space is called \emph{large-scale proximity} if
\begin{enumerate}
\item if $B\s X$ then $B\bar r B$ if and only if $B$ is bounded;
\item  $A r B$ implies $B r A$ for every $A,B\s X$;
\item if $A,A'\s X$ are subsets and $E\s X^2$ is an entourage with $E[A]\z 
A',E[A']\z A$ then $A r B$ implies $A' r B$ for every $B\s X$; 
\item if $A,B,C\s X$ then $(A\cup B)r C$ if and only if ($A r C$ or $B r C$);
\item if $A,B\s X$ with $A\bar r B$ then there exist $C,D\s X$  with $C\cup D=X$ 
and $C\bar r A,D\bar r B$.
\end{enumerate}
\end{defn}

\begin{rem}
Note a large-scale proximity relation on a metric space is an instance of a coarse proximity relation as defined in \cite[Definition~2.2]{Grzegrzolka2018a}. Compare this 
notion with the notion of proximity relation \cite{Naimpally1970},\cite{Willard1970}. The axiom 3 of Definition~\ref{defn:closerelation} is the characteristic for our application on coarse metric spaces.
\end{rem}

\begin{lem}
\label{lem:finerthanclose}
Let $X$ be a metric space. Every large-scale proximity relation on $X$ is finer than the relation close $\close$.
\end{lem}
\begin{proof}
Let $r$ be a large-scale proximity relation on a metric space $X$. If $A\close B$ then there  exist unbounded $(a_i,b_i)_i\s X^2$ and an entourage $E\s X^2$ such that 
$E[(a_i)_i]\z (b_i)_i$ and $E[(b_i)_i]\z (a_i)_i$. This implies $(a_i)_i r 
(b_i)_i$ by axiom 3 of Definition~\ref{defn:closerelation}. By axiom 4 of 
Definition~\ref{defn:closerelation} the relation $A r B$ holds.
\end{proof}

Now we construct a topological space $\partial_r' X$ given a large-scale proximity relation $r$. This will turn out to be the boundary of a coarse compactification. The topology on this construction is easier to describe than in the other equivalent definition which will follow below.

\begin{defn}
Let $r$ be a large-scale proximity relation on a metric space $X$. A system 
$\sheaff$ of subsets of $X$ is called an $r$-ultrafilter if
\begin{enumerate}
      \item $A,B\in\sheaff$ implies $A r B$;
      \item if $A,B\s X$ are subsets with $A\cup B\in \sheaff$ then $A\in 
\sheaff$ or $B\in \sheaff$;
\item $X\in\sheaff$.
\end{enumerate}
Denote by $\delta_r X$ the set of $r$-ultrafilters. If $A\s X$ is a subset then 
define 
\[
     \closedop A:=\{\sheaff\in \delta_r X: A\in\sheaff\}.
\]
\end{defn}

\begin{lem}
If $X$ is a metric space the $(\closedop A ^c)_{A\s X}$ constitute a base for a topology on $\delta_r X$.
\end{lem}
\begin{proof}
First we show the base elements cover $\delta_r X$: Since $\emptyset$ is bounded $\emptyset \bar r X$. This implies $\emptyset\not \in \sheaff$ for every $r$-ultrafilter $\sheaff$. Thus $\delta_r X=\closedop \emptyset^c$.

Now we show for every element in the intersection of two base elements there is 
a base element which contains the element and is contained in the intersection: 
Let $A,B\s X$ be two subsets. Let $\sheaff\in \closedop A^c\cap \closedop B^c$ 
be an element. Then $A\not\in \sheaff, B\not \in \sheaff$ thus $(A\cup 
B)\not\in \sheaff$. This implies $\sheaff \in \closedop {A\cup B}^c\s 
\closedop A^c\cap \closedop B^c$.
\end{proof}

\begin{defn}
\label{defn:relationtocompactification1}
Define the topology on $\delta_r X$ to be the topology generated by $(\closedop A^c)_{A\s X}$.

Now define a relation $\lambda_r$ on $\delta_r X$: $\sheaff\lambda_r \sheafg$ 
if $A\in \sheaff,B\in \sheafg$ implies $A r B$. The quotient by this equivalence relation $\partial'_r X=\delta_r X/\lambda_r$ is called \emph{$r$-boundary 1}. 
\end{defn}

\begin{lem}
The space $\partial'_r(X)$ is a compact Hausdorff topological space.
\end{lem}
\begin{proof}
The proof of \cite[Theorem 26]{Hartmann2019a} with $\close$ replaced by $r$ 
implies that $\partial'_r (X)$ is compact.

Now we show $\partial'_r(X)$ is Hausdorff. Let $\sheaff,\sheafg$ be two 
$r$-ultrafilters with $\sheaff\bar\lambda_r \sheafg$. Thus there exist $A\in 
\sheaff,B\in \sheafg$ with $A\bar r B$. Then there exist $C,D\s X$ with $C\cup 
D=X$ and $C\bar r A, D\bar r B$. Then $\sheafg\in\closedop D^c,\sheaff \in 
\closedop C^c$. Also:
\f{
\closedop C^c\cap \closedop D^c
&=(\closedop C \cup \closedop D)^c\\
&=\closedop X^c\\
&=\emptyset.
}
\end{proof}

This completes our discussion of $\partial_r'(X)$. We now define another topological space $\partial_r X$ given a large-scale proximity relation. This space is homeomorphic to $\partial_r' X$. Compared with the previous model the points on $\partial_r X$ are easier to describe.

Let $R\ge0$ be a real number. A metric space $X$ is called $R$-discrete if $d(x,y)\ge R$ for every $x\not=y$. If $X$ is a metric space an $R$-discrete for some $R>0$ subspace $S\s X$ is called a Delone set if the inclusion $S\to X$ is coarsely surjective. Every metric space contains a Delone set. 

\begin{defn}
Let $r$ be a large-scale proximity relation on a proper metric space $X$ and $S\s X$ a Delone subset. Denote by $\hat S$ the set of nonprincipal ultrafilters on 
$S$. If $A\s S$ is a subset define
\[
\closedop A:=\{\sheaff\in \hat S:A\in \sheaff\}.
\]
Then define a relation $r$ on subsets of $\hat S$: $\pi_1 r \pi_2$ if for every 
$A,B\s S$ the relations $\pi_1\s \closedop A,\pi_2\s \closedop B$ imply $A r 
B$.
\end{defn}

\begin{lem}
The relation $r$ on $\hat S$ is a proximity relation.
\end{lem}
\begin{proof}
The proof of \cite[Theorem~23]{Hartmann2019a} with $r$ in place of $\close$ 
applies.
\end{proof}

As before we define a topology on $\hat S$:
\begin{defn}
The relation $r$ on subsets of $\hat S$ determines a Kuratowski closure operator
\[
\bar \pi=\{\sheaff\in \hat S:\{\sheaff\} r \pi\}.
\]
Now define a relation $\lambda_r$ on $\hat S$: $\sheaff\lambda_r \sheafg$ if 
$A\in \sheaff,B\in \sheafg$ implies $A r B$. 
\end{defn}

\begin{lem}
The relation $\lambda_r$ is an equivalence relation on $\hat S$. 
\end{lem}
\begin{proof}
The relation is obviously symmetric and reflexive. We show transitivity. Let $\sheaff_1,\sheaff_2,\sheaff_3$ be nonprincipal ultrafilters on $X$ such that $\sheaff_1\lambda_r \sheaff_2$ and $\sheaff_2\lambda_r \sheaff_3$. We show $\sheaff_1\lambda_r \sheaff_3$. Assume the opposite. There are $A\in \sheaff_1,B\in \sheaff_3$ with $A \bar r B$. Then there exist $C,D\s X$ with $C\cup D=X$ and $C\bar r A,D\bar r B$. Now $C\in \sheaff_2$ or $D\in \sheaff_2$. If $C\in \sheaff_2$ this contradicts $\sheaff_2\lambda_r \sheaff_1$ and if $D\in \sheaff_2$ this contradicts $\sheaff_2\lambda_r \sheaff_3$.
\end{proof}

\begin{defn}
\label{defn:relationtocompactification2}
 Now the \emph{$r$-boundary 2} is defined $\partial_r(X)=\hat S/\lambda_r$ as the quotient by $\lambda_r$.
\end{defn}
 
We check this definition does not depend upon the choice of Delone set $S$:

\begin{lem}
If $T\s X$ is another Delone subset then $\hat S/\lambda_r=\hat 
T/\lambda_r$ are homeomorphic.
\end{lem}
\begin{proof}
Suppose $S,T\s X$ are two Delone sets. Without loss of generality 
assume $S\s T$ is a subset. Then there exists a map $\varphi:T\to S$ 
with $H:=\{(t,\varphi(t)):t\in T\}$ an entourage and $\varphi\circ i=id_S$ 
where $i:S\to T$ is the inclusion. There is an induced map
\f{
\varphi_*:\hat T&\to \hat S\\
\sheaff&\mapsto \{A\s S:\iip \varphi A\in \sheaff\}.
}
We show $\varphi_*$ is continuous and respects $\lambda_r$. If $\pi_1,\pi_2\s 
\hat T$ are subsets with $\pi_1 r \pi_2$ and $A,B\s S$ are subsets with 
$\varphi_*\pi_1\s \closedop A,\varphi_*\pi_2\s \closedop B$ then $\pi_1\s 
\closedop {\iip \varphi A},\pi_2\s \closedop{\iip \varphi B}$. Thus $\iip 
\varphi A r \iip \varphi B$. Since $H[\iip \varphi A]=A,\ii H[A]=\iip\varphi A$ 
and $H[\iip \varphi B]=B,\ii H[B]=\iip \varphi B$ this implies $A r B$. Thus 
$\varphi_*\pi_1 r \varphi_* \pi_2$.

Let $\sheaff,\sheafg\in \hat T$ be elements with $\sheaff\lambda_r \sheafg$. 
Let $A\in \varphi_*\sheaff,B\in\varphi_*\sheafg$ be elements. Then $\iip 
\varphi A\in \sheaff,\iip \varphi B\in\sheafg$. Thus $\iip \varphi A r \iip 
\varphi B$. This implies $A r B$ by the above.

Now the inclusion $i$ induces a map $i_*\hat S \to \hat T$ in a similar 
way. We show just as for $\varphi_*$ that $i_*$ is continuous and respects 
$\lambda_r$. We have $\varphi_*\circ i_*(\sheaff)=\sheaff$ for every 
$\sheaff\in \hat S$ and $i_*\circ \varphi_*(\sheafg)\lambda_r \sheafg$ for 
every $\sheafg\in \hat T$ since $A,B\in \sheafg$ implies $A r B$ which implies 
$\iip \varphi{A\cap S} r B$. 
\end{proof}

Comparing the first and the second model, we can prove:
\begin{prop}
If $X$ is a proper metric space the map 
      \f{
      \Phi:\partial_rX&\to \partial_r' X\\
      [\sigma]&\mapsto [\{A\s X: A r B \forall B\in \sigma\}]
      }
is a homeomorphism. 
\end{prop}
\begin{proof}
Let  $S\s X$ be a Delone subset. If $\sigma$ is an ultrafilter on $S$ 
then the collection $\{A\s X:A r B \forall B\in \sigma\}$ is an $r$-ultrafilter 
on $X$ by a proof similar to that of \cite[Theorem~17]{Hartmann2019a}. If 
$\sigma,\tau$ are two ultrafilters on $S$ with $\sigma \lambda_r \tau$ then 
$\sigma,\tau$ are in particular $r$-ultrafilters on $S$ and 
$\sigma\lambda_r(\Phi(\sigma))|_S, \tau\lambda_r(\Phi(\tau))|_S$. By 
transitivity of $\lambda_r$ we obtain 
$(\Phi(\sigma))|_S\lambda_r(\Phi(\tau))|_S$. This implies 
$\Phi(\sigma)\lambda_r\Phi(\tau)$ Thus $\Phi$ is well defined.   
 
 Now we show $\Phi$ is injective: Let $\sigma,\tau$ be nonprincipal 
ultrafilters on $S$ with $\Phi(\tau)\lambda_r\Phi(\sigma)$. Then 
$\sigma\lambda_r\tau$ since $\tau\s \Phi(\tau),\sigma\s\Phi(\sigma)$.

Now we show $\Phi$ is surjective: Let $\sheaff$ be an $r$-ultrafilter on $X$. 
Without loss of generality assume $S\in \sheaff$. Then by 
\cite[Lemma~5.3]{Naimpally1970} there exists an ultrafilter $\sigma$ on $X$ 
such that $S\in \sigma$ and $\sigma\s \sheaff$. Then $\sigma|_S$ is mapped by 
$\Phi$ to the class of $\sheaff$. 

If $A\s S$ is a subset we show $\closedop A:=\{[\sigma]:A\in\sigma\}$ is a 
closed subset of $\partial_r X$. Let $\sigma\in \hat S$ be an element with 
$\sigma r \closedop A$. Then for every $B\in \sigma$ we obtain $Br A$. Thus 
$A\in\Phi(\sigma)$. By \cite[Lemma~5.7]{Naimpally1970} there exists an 
ultrafilter $\tau$ on $S$ with $\tau\s \Phi(\sigma)$ and $A\in\tau$. This 
implies $\tau\lambda_r \sigma$. Thus $[\sigma]\in \closedop A$. In fact the 
$(\closedop A^c)_{A\s S}$ constitute a base for the topology on 
$\partial_r(X)$.

If $A\s S$ is a subset then $\Phi(\closedop A)=\closedop A$. Thus $\Phi$ is a 
closed map.

If $A\s X$ then there exists an entourage $E\s X^2$ such that $E[A]\s S$ and 
$A,E[A]$ are finite Hausdorff distance apart. Then $\iip \Phi {\closedop 
A}=\closedop {E[A]}$. Thus $\Phi$ is continuous.
\end{proof}

\begin{prop}
\label{prop:quotientclosefiner}
If $r,s$ are two large-scale proximity relations on a proper metric space $X$ and $s$ is finer than $r$ then there is a quotient map 
\[
\partial_r(X)\to \partial_s(X).
\]
\end{prop}
\begin{proof}
Let $S\s X$ be a Delone subset.

If $\sheaff,\sheafg$ are nonprincipal ultrafilters on $S$ then $\sheaff\lambda_r\sheafg$ implies ($A\in \sheaff,B\in\sheafg$ implies $A r B$). Thus $A s B$ for every $A\in \sheaff, B\in \sheafg$ which implies $\sheaff\lambda_s \sheafg$.

Now we show $id_{\hat S}:(\hat S,r)\to (\hat S,s)$ is continuous. If $\pi_1,\pi_2\s \hat S$ are subsets then $\pi_1 r \pi_2$ implies ($A,B\s S$ with $\pi_1\s \closedop A, \pi_2\s \closedop B$ implies $A r B$). Then $A s B$ if $\pi_1\s \closedop A, \pi_2\s \closedop B$. Thus $\pi_1 s \pi_2$.

Since $id_{\hat S}$ is surjective the induced map on quotients is surjective.

Since for every subset $A\s S$ the map $id_{\hat S}$ maps $\closedop A$ to $\closedop A$ the induced map on quotients is closed.
\end{proof}

Now we produce the compactification of a proper metric space $X$ with the boundary $\partial_r X$ given a proximity relation $r$. Define $\bar X^r=X\sqcup \partial_r X$ as a set. Closed sets on $\bar X^r$ are generated by $(\bar A\cup \closedop A)_{A\s X}$, where the closure $\bar A$ of $A$ is taken in $X$.

\begin{prop}
If $X$ is a proper metric space and $r$ a large-scale proximity relation on $X$ then $\bar X^r$ is a compactification of $X$ with boundary $\partial_r X$.
\end{prop}
\begin{proof}
. This topology is compact by the first part of the proof of \cite[Theorem~20]{Hartmann2019c} with $\close$ replaced by $r$. The spaces $\partial_r X,X$ appear as subspaces of $\bar X^r$. The inclusion $X\to \bar X^r$ is dense since $X\cup \closedop X=\bar X^r$.
\end{proof}

\begin{rem}
\label{rem:uniquequotient}
The statement in Proposition~\ref{prop:quotientclosefiner} can be 
strengthened: If $X$ is a proper metric space and $r,s$ are close relations on 
$X$ with $s$ finer than $r$ then there is a quotient map $\bar X^r\to \bar X^s$ which is the unique continuous map extending the identity on $X$.
\end{rem}
\begin{proof}
Assume without loss of generality that $X$ is $R$-discrete for some $R>0$. A net in $X$ converging to a point $p\in \bar X$ can be written as a filter $\sheaff$ on $X$. Then an ultrafilter $\sigma$ finer than $\sheaff$ converges to the same point.

Let $\alpha:\bar X^r\to \bar X^s$ be a continuous map extending the identity on 
$X$. Then $\alpha_*\sigma=\sigma$. Thus $\sigma$ converges to $\alpha(p)$. Now $p$ is represented by $\sigma$ in $\partial_r X$ and $\alpha(p)$ is also represented by $\sigma$ in $\partial_s X$. Thus $\alpha$ maps a point represented by $\sigma$ to a point represented by $\sigma$ and is thus uniquely determined. This implies $\alpha|_{\partial_r X}$ is the quotient map of Theorem~\ref{prop:quotientclosefiner}.
\end{proof}

\begin{rem}
Remark~\ref{rem:uniquequotient} and Lemma~\ref{lem:finerthanclose} imply that the Higson compactification is universal among coarse compactifications. This recovers \cite[Proposition~2.39]{Roe2003}.
\end{rem}

If $\bar X$ is a coarse compactification of a proper metric space $X$, define a relation $r_{\bar X}$ on subsets of $X$ as follows: For a subset $A\s X$ define $\closedop A=\bar A\cap \partial X$. If $A,B\s X$ are subsets then $Ar_{\bar X} B$ if $\closedop A\cap \closedop B\not=\emptyset$.

\begin{thm}
\label{thm:coarsetorelation}
Let $\bar X$ be a coarse compactification of a proper metric space $X$. Then $r_{\bar X}$ is a large-scale proximity relation on $X$. There is a homeomorphism $\Phi:\bar X\to \bar X^r$ extending the identity on $X$.
\end{thm}
\begin{proof}
      By \cite[Theorem~6.7]{Grzegrzolka2018a} the relation $r$ is a coarse 
proximity relation. Thus axioms 1,2,4,5 of a large-scale proximity relation are satisfied. It remains to show $r$ satisfies axiom 3. Let $A,A',B\s X$ be subsets and let $E\s X\times X$ be an entourage with $E[A]\z A',E[A']\z A$ and $Ar B$. Then 
$\closedop A\cap \closedop B\not=\emptyset$. Since the compactification is 
coarse $A,A'$ have the same limit points on $\partial X$. Thus $\closedop 
A=\closedop {A'}$ which implies $\closedop {A'}\cap \closedop B\not=\emptyset$. 
Thus $A'r B$.

For the last statement we extend the proof of \cite[Theorem~6.7]{Grzegrzolka2018a}. If $x\in \bar X\setminus X$ is a point define $\sheaff_x:=\{A\s  X: x\in \closedop A\}$. Then we define
\begin{align*}
\Phi:\bar X&\to \bar X^r\\
x&\mapsto \begin{cases}
           x & x\in X\\
           [\sheaff_x] & x\in \bar X\setminus X
          \end{cases}
\end{align*}
Now \cite[Theorem~6.7]{Grzegrzolka2018a} showed $\Phi|_{\bar X\setminus X}$ is 
a homeomorphism. This implies in particular that $\Phi$ is a bijective map. We 
show $\Phi$ is continuous: Let $A\s X$ be a subset. Then
\f{
\iip \Phi {\bar A^X\cup \closedop A}
&=\bar A^X\cup \{x\in \bar X\setminus X:\sheaff_x\in\closedop A\}\\
&=\bar A^X\cup \{x\in\bar X\setminus X: A\in \sheaff_x\}\\
&=\bar A^X\cup \closedop A
}
is a closed set. Here $\bar A^X$ denotes the closure of $A$ in $X$.
\end{proof}

%% file: Bounded_Functions.tex
\section{Bounded functions}

\begin{thm}
\label{thm:coarsefunctions}
Let $X$ be a proper metric space. A compactification $\bar X$ is coarse if and 
only if every bounded continuous function $\varphi:X\to \R$ that extends to 
$\bar X$ is Higson.
\end{thm}
\begin{proof}
Suppose $\bar X$ is a coarse compactification and assume for contradiction there is a continuous function $\varphi:\bar X\to \R$ such that $\varphi|_X$ is not Higson. Then there is an entourage $E\s X^2$ and some $\varepsilon>0$ and an unbounded sequence $(x_k, y_k)_k\s E$ with 
\[
|\varphi(x_k)-\varphi(y_k)|>\varepsilon
\]
Now $\overline{(x_k)_k}\cap \overline{(y_k)_k}\not=\emptyset$ since the compactification is coarse. This contradicts that $\varphi(x_k)_k, \varphi(y_k)_k$ have disjoint limit points in $\R$.

Now we give an alternative proof of this direction using that the Higson corona is universal among coarse compactifications. Suppose $\varphi:X\to \R$ is a bounded continuous function that extends a continuous function $\bar \varphi$ on $\bar X$. Denote by $q:hX\to \bar X$ the quotient map from the Higson corona. Then $\bar \varphi\circ q:hX\to \R$ is an extension of $\varphi$ to $hX$. This implies $\varphi$ is Higson.

Now suppose every bounded continuous function $\varphi:X\to \R$ that extends to $\bar X$ is Higson. Let $A,B\s X$ be subsets such that $\bar A\cap \bar B=\emptyset$. Since $\bar X$ is normal we can use Urysohn's lemma: there exists a bounded continuous function $\varphi:\bar X\to \R$ with $\varphi|_{\bar A}\equiv 0$ and $\varphi|_{\bar B}\equiv 1$. Now $\varphi|_X$ is Higson. this implies for every entourage $E\s X\times X$ there exists a bounded set $C\s X$ with
\[
      E\cap (A\times B)\s C\times C.
\]
This implies $A\notclose B$.
\end{proof}

Let $X$ be a proper metric space. To a coarse compactification $\bar X$ we can associate a large-scale proximity relation $r$ on subsets of $X$ such that $\bar X=\bar X^r$. We can also associate to $\bar X$ the set of bounded functions $C_r(X)$ that extend to $\bar X^r$. They must be Higson. Note that $C_r(X)$ is a ring by pointwise addition and multiplication.

%% file: Functoriality.tex
\section{Functoriality}

\begin{prop}
\label{prop:contravariant}
Let $\alpha:X\to Y$ be a coarse map between proper metric spaces and let $\bar Y^r$ be a coarse compactification. If $X$ is $R$-discrete for some $R>0$ the functions $C_r(X):=\{\varphi\circ \alpha:\varphi\in C_r(Y)\}$ determine a coarse compactification on $X$. It is the same compactification which is induced by the relation $Ar B$ if $\alpha(A)r\alpha(B)$.
\end{prop}
\begin{proof}
If $\varphi$ is a Higson function on $Y$ then $\varphi\circ \alpha$ is 
continuous and bounded since $\alpha$ is continuous and $\varphi\circ\alpha$ is continuous and bounded. Since $\alpha$ is a coarse map $\varphi\circ\alpha$ satisfies the Higson property.  Thus $C_r(X)$ determines a compactification $\bar X$ which is coarse. The set $C_r(X)$ is a ring by pointwise addition and multiplication and contains the constant functions. Thus $C_r(X)$ equals the bounded functions which can be extended to $\bar X^r$.

Now we prove the relation $r$ defined on subsets of $X$ is a large-scale proximity relation. We check the axioms of a large-scale proximity relation:
\begin{enumerate}
      \item if $B\s X$ is bounded so is $\alpha(B)\s Y$. Thus $\alpha(B)\bar r 
\alpha(B)$ which implies $B\bar r B$. If $A\s X$ is unbounded then $\alpha(A)\s 
Y$ is unbounded. Thus $\alpha(A)r \alpha(A)$ which implies $Ar A$.
    \item symmetry is obvious.
    \item Suppose $A,A',B\s X$ are subsets and $E\s X\times X$ an entourage 
with $E[A]\z A',E[A']\z A$ and $Ar B$. Then $\zzp \alpha 
E[\alpha(A)]\z\alpha(A')$ and $\zzp \alpha E[\alpha(A')]\z\alpha(A)$ and 
$\alpha(A)r\alpha(B)$. Then $\alpha(A')r\alpha(B)$. Thus $A'r B$.
    \item If $(A\cup B)r C$ then $\alpha(A\cup B)r \alpha(C)$. Now 
$\alpha(A\cup B)=\alpha(A)\cup \alpha(B)$ thus $\alpha(A)r \alpha(C)$ or 
$\alpha(B)r \alpha (C)$. This implies $Ar C$ or $Br C$.
    \item If $A\bar r B$ then $\alpha(A)\bar r \alpha(B)$. This implies there 
exist $C,D\s Y$ with $C\cup D=Y$ and $C\bar r \alpha(A),D\bar r \alpha(B)$. 
Then $\iip \alpha C\cup \iip \alpha D=X$ and $\iip \alpha C\bar r A,\iip \alpha 
D\bar r B$.
\end{enumerate}

Now we define a map
\begin{align*}
      \Phi:\bar X^r&\to \R^{C_r(X)}\\
      x&\mapsto\begin{cases}
                     (\varphi\circ \alpha(x))_{\varphi\circ \alpha} & x\in X\\
                     (x\mhyphen\lim \varphi\circ \alpha)_{\varphi\circ \alpha} 
& x\in \partial_r X.
               \end{cases}
\end{align*}
We show $\Phi$ is well-defined: If $\sheaff\lambda_r \sheafg$ then for every $A\in \sheaff,B\in \sheafg$ the relation $Ar B$ holds. Then $\alpha(A)r \alpha(B)$ which implies $\alpha_*\sheaff \lambda_r \alpha_*\sheafg$. Then
\f{
\sheaff\mhyphen\lim \varphi\circ \alpha
&=\alpha_*\sheaff\mhyphen\lim \varphi\\
&=\alpha_*\sheafg\mhyphen\lim \varphi\\
&=\sheafg\mhyphen\lim \varphi\circ \alpha
}
for every $\varphi\in C_r(Y)$.

Now we show $\Phi$ is injective: if $\sheaff\bar \lambda_r \sheafg$ on $X$ then there exist $A\in \sheaff, B\in \sheafg$ with $A\bar r B$. Thus $\alpha(A)\bar r \alpha(B)$. This implies $\alpha_*\sheaff\bar \lambda_r \alpha_*\sheafg$. Then there exists some bounded function $\varphi\in C_r(Y)$ with
\f{
\sheaff\mhyphen\lim \varphi\circ \alpha
&=\alpha_*\sheaff\mhyphen\lim \varphi\\
&\not=\alpha_*\sheafg\mhyphen\lim \varphi\\
&=\sheafg\mhyphen\lim \varphi\circ \alpha.
}
Denote for $Z=X,Y$ the evaluation map 
\f{
e:Z&\to \R^{C_r(Z)}\\
z&\mapsto(\varphi(z))_\varphi.
}

Now the following diagram commutes
\[
  \xymatrix{
  \bar X^r\ar[r]^{\alpha_*}\ar[d]_\Phi
  &\bar Y^r\ar@{=}[d]\\
  \overline{e(X)}\ar[r]_{\alpha_*}
  &\overline{e(Y)}
  }
\]
where the upper horizontal map maps $\sheaff\in\partial_r X$ to $\alpha_*\sheaff$ and $x\in X$ to $\alpha(x)$ and the lower horizontal map maps $(\varphi\circ\alpha(x))_{\varphi\circ\alpha}$ to $(\varphi\circ\alpha(x))_\varphi$. Now both horizontal maps are continuous and open, thus $\Phi$ is continuous and open.

Since every ultrafilter on $X$ induces an $r$-ultrafilter on $X$ the map $\Phi$ is surjective on $\overline{e(X)}$.
\end{proof}

\begin{lem}
      Let $\alpha:X\to Y$ be a coarse map between proper metric spaces and let
$\bar X^{r_1},\bar Y^{r_2}$ be coarse compactifications of $X,Y$, respectively. 
If $Ar_1 B$ implies $\alpha(A)r_2\alpha(B)$ then $\alpha$ can be extended to a 
continuous map
\begin{align*}
\alpha_*:\bar X^{r_1} &\to \bar Y^{r_2}\\
x&\mapsto\begin{cases}
               \alpha(x) &x\in X\\
               \alpha_*x & x\in \partial_{r_1}X.
         \end{cases}
\end{align*}
If $X$ is $R$-discrete for some $R>0$ and $C_{r_2}(Y)\circ\alpha\s C_{r_1}(X)$ 
then $\alpha_*$ is the restriction of
\f{
\R^{C_{r_1}(X)}&\to \R^{C_{r_2}(Y)}\\
(\varphi(x))_\varphi&\mapsto (\varphi\circ\alpha(x))_\varphi
}
to $\overline{e(X)}$. Both descriptions of $\alpha$ coincide.
\end{lem}
\begin{proof}
      Suppose $A r_1 B$ implies $\alpha(A)r_2\alpha(B)$. Let $\varphi\in 
C_{r_2}(Y)$ be a function and let $\sheaff$ be an $r_1$-ultrafilter on $X$. 
Then $\alpha_*\sheaff$ is an $r_2$-ultrafilter on $Y$. Thus 
$\alpha_*\sheaff\mhyphen\lim \varphi$ exists. This point equals 
$\sheaff\mhyphen\lim\varphi\circ\alpha$. Since $\sheaff$ was arbitrary the map 
$\varphi\circ\alpha$ can be extended to $\bar X^{r_1}$. Thus 
$\varphi\circ\alpha\in C_{r_1}(X)$.

Now suppose $C_{r_2}(Y)\circ \alpha\s C_{r_1}(X)$. Let $A,B\s Y$ be subsets 
with $A \bar r_2 B$. Then there exists $\varphi\in C_{r_2}(Y)$ with 
$\varphi|_A\equiv 1,\varphi|_B\equiv 0$. Then $\varphi\circ\alpha|_{\iip \alpha 
A}\equiv 1,\varphi\circ\alpha|_{\iip \alpha B}\equiv 0$ and $\varphi\circ\alpha$ can be extended to $\bar X^{r_1}$. Thus $\iip \alpha A \bar r_1 \iip \alpha B$.

Note the diagram
\[
  \xymatrix{
  \bar X^{r_1}\ar[r]^{\alpha_*}\ar@{=}[d]
  &\bar Y ^{r_2}\ar@{=}[d]\\
  \overline{e(X)}\ar[r]_{\alpha_*}
  &\overline{e(Y)}
}
\]
commutes.
\end{proof}

\begin{prop}
\label{prop:ccsheaf}
Let $X$ be a proper metric space. If subsets $U_1,\ldots,U_n$ coarsely cover $X$ and each $U_i$ is equipped with a large-scale proximity relation $r_i$ such that $r_i,r_j$ agree on $U_i\cap U_j$ then the relation $r$ on subsets of $X$ defined by $Ar B$ if $(U_i\cap A) r_i(U_i\cap B)$ for some $i$ defines a large-scale proximity relation on $X$. If $X$ is $R$-discrete for some $R>0$ and for every $i$ there is a ring $C_{s_i}(U_i)$ of Higson functions such that $C_{s_i}(U_i)|_{U_j}=C_{s_j}(U_j)|_{U_i}$ then the ring 
\[
      C_s(X)=\{(\varphi_i)_i\in\prod_i 
C_{s_i}(U_i):\varphi_i|_{U_j}=\varphi_j|_{U_i}\}
\]
consists of Higson functions. If $r_i=s_i$ for every $i$ then the relation $r$ and the ring of bounded functions $C_s(X)$ describe the same compactification.
\end{prop}
\begin{proof}
      We show $r$ is a large-scale proximity relation on $X$:
      \begin{enumerate}
            \item if $B\s X$ is bounded then $B\cap U_i$ is bounded for every $i$. Thus $B\bar r B$. If $A\s X$ is not bounded then there exists some $i$ such that $A\cap U_i$ is not bounded. Then $(A\cap U_i)r_i (A\cap U_i)$ thus $ArA$.
            \item Symmetry is obvious.
            \item Without loss of generality assume $n=2$. Let $A,A',B\s X$ be subsets and let $E\s X\times X$ be an entourage with $E[A]= A',E^{-1}[A']= A$ and 
$Ar B$. Since $U_1,U_2$ coarsely cover $X$ the relation $U_1^c\notclose U_2^c$ holds. Thus there exist $C,D\s X$ with $C\cup D=X$ and $C\notclose U_1^c,D\notclose U_2^c$. Now $A=(A\cap C)\cup (A\cap D)$. Thus by axiom 4 $(A\cap C)r B$ or $(A\cap D)r B$. Suppose the former holds. Since $A\cap C\s U_1\cup B'$ where $B'$ is bounded we have $(A\cap C\cap U_1)r_1 (B\cap U_1)$. Now $E[A\cap C]\s U_1\cup B''$ where $B''$ is bounded. Then $(E[A\cap C]\cap U_1)r_1 (B\cap U_1)$. Thus $A' r B$ by axiom 4.             \item If $(A\cup B)r C$ then $(A\cup B)r_i C$ for some $i$. Thus $A 
r_i C$ or $B r_i C$. This implies $Ar C$ or $Br C$. If $Ar C$ or $Br C$ then 
$(A\cap U_i)r_i (C\cap U_i)$ for some $i$ or $(B\cap U_j)r_j(C\cap U_j)$ for 
some $j$. This implies $((A\cup B)\cap U_i)r_i(C\cap U_i)$ or $((A\cup B)\cap 
U_j)r_j(C\cap U_j)$. Thus $(A\cup B)r C$.
            \item Without loss of generality assume $n=2$ and $U_1,U_2$ cover 
$X$ as sets. If $A\bar r B$ then $(A\cap U_i)\bar r_i (B\cap U_i)$ for both 
$i$. Thus there exist $C_i,D_i\s U_i$ with $C_i\bar r_i (A\cap U_i),D_i\bar r_i 
(B\cap U_i)$ and $C_i\cup D_i=U_i$ for $i=1,2$. Then $(A\cap U_1\cap U_2)\bar r 
(C_1\cup C_2), (A\cap U_1\cap U_2^c) \bar r (C_1\cup U_2),(A\cap U_2\cap 
U_1^c)\bar r (C_2\cup U_1)$ combine to $A\bar r (C_1\cup C_2)$. Similarly we 
obtain $B\bar r (D_1\cup D_2)$. Now
\f{
X
&=U_1\cup U_2\\
&=C_1\cup D_1\cup C_2\cup D_2.
}
      \end{enumerate}
Now we show $C_s(X)$ consists of Higson functions. Suppose $\varphi_i\in 
C_{s_i}(U_i)$ for $i=1,\ldots,n$ are elements with 
$\varphi_i|_{U_j}=\varphi_j|_{U_i}$. Then they can be glued to a bounded 
continuous function $\varphi:X\to \R$. Let $E\s X\times X$ be an entourage. Then
\[
      E=(E\cap(U_1\times U_1))\cup\cdots\cup (E\cap(U_n\times U_n))\cup A
\]
where $A\s B\times B$ with $B$ bounded in $X$. Now 
$(d\varphi)|_{E\cap(U_i\times U_i)}=(d\varphi|_{U_i})|_E$ converges to zero at 
infinity for every $i$. This implies $(d\varphi)|_E$ converges to zero at 
infinity.

If $r_i=s_i$ for every $i$ then $\bar U_i^{r_i}=\overline{e(U_i)}$ for every 
$i$. The $\bar U_i^{r_i}$ glue to $\bar X^r$ and the $\overline{e(U_i)}$ glue 
to $\overline{e(X)}$. The global axiom of Lemma~\ref{lem:ccsheaf} implies 
uniqueness. Thus $\bar X^r=\overline{e(X)}$.

\end{proof}

Let $X$ be a proper $R$-discrete for some $R>0$ metric space. For every  subspace $A\s X$ the poset of coarse compactifications on $A$ is called $\mathcal{CC}(A)$. If $A\s B$ is an inclusion of subspaces then there is a poset map $\mathcal{CC}(B)\to\mathcal{CC}(A)$ induced by the inclusion.
 
The Grothendieck topology determined by coarse covers on a metric space $X$ is called $X_{ct}$. A contravariant functor $\sheaff$ on subsets of $X$ is a sheaf on $X_{ct}$ if for every coarse cover $U_1,U_2\s U$ of a subset of $X$ the following diagram is an equalizer
\[
 \sheaff(U)\to\sheaff(U_1)\oplus\sheaff(U_2)\rightrightarrows\sheaff(U_1\cap U_2).
\]

\begin{lem}
\label{lem:ccsheaf}
 The functor $\mathcal {CC}$ on subsets of $X$ is a sheaf on $X_{ct}$. 
\end{lem}
\begin{proof}
Note every subspace of $X$ is proper. If $A\s X$ is a subset we define $\bar A^r\ge \bar A^s$ if $s$ is finer than $r$.
      
If $A\s B$ is an inclusion of subspaces and $\bar B^r\in\mathcal{CC}(B)$ then the restriction map associated to the inclusion $A\to B$ maps $\bar B^r\mapsto \bar A^{r|_A}$. Here the relation $r|_A$ is defined as $Sr|_AT$ if $SrT$. Then $r|_A$ is a large-scale proximity relation on $A$:
\begin{enumerate}
\item If $S\s A$ is bounded, then $S\bar r S$, thus $S\bar r|_A S$. If $S$ is unbounded then $S r S$ so $S r|_A S$. 
\item Symmetry is obvious.
\item If $S,S',T\s A$ are subsets, $E\s A^2$ is an entourage with $E[S]\z S',E[S']\z S$ and $Sr|_AT$ then $E\s B^2$ is an entourage in $B$. Thus $S'r T$ which implies $S'r|_AT$.
\item obvious.
\item If $S,T\s A$ are subsets with $S\bar r T$ then there exist subsets $C',D'\s B$ with $C'\cup D'=B$ and $C'\bar r S,D'\bar r T$. Then $C:=C'\cap A,D:=D'\cap A$ are subsets with $C\cup D=A, C\bar r S,D\bar r T$.
\end{enumerate}
Note if a large scale proximity relation $s$ on $B$ is finer than another large-scale proximity relation $r$ on $B$ then $s|_A$ is finer than $r|_A$ on $A$. This makes $\mathcal{CC}$ into a functor on the poset of subsets of $X$ to posets.

Now we check the global axiom: Let $(U_i)_i$ be a coarse cover of $X$ and let $r,s$ be close relations on $X$ with $r|_{U_i}=s|_{U_i}$. Two subsets $A,B\s X$ satisfy $ArB$ if and only if $\bigvee_i((A\cap U_i)r(B\cap U_i))$ if and only if $\bigvee_i((A\cap U_i)r|_{U_i}(B\cap U_i))$ if and only if $\bigvee_i((A\cap U_i)s|_{U_i}(B\cap U_i))$ if and only if $\bigvee_i((A\cap U_i)s(B\cap U_i))$ if and only if $AsB$.

Now we check the gluing axiom: Let $U_1,\ldots,U_n$ be a coarse cover of $X$ equipped with coarse compactifications $\bar U_1^{r_1},\ldots , \bar U_n^{r_n}$ such that $r_i|_{U_j}=r_j|_{U_i}$ for every $ij$. Then the proof Proposition~\ref{prop:ccsheaf} implies the $\bar U_i^{r_i}$ glue to a coarse compactification $\bar X^r$ of $X$.
\end{proof}

%% file: Higson.tex
\section{Higson corona}

This section is denoted to the Higson corona. We recall the original description.
\begin{defn}\name{Higson corona}
Let $X$ be a proper metric space. A bounded continuous function $\varphi:X\to 
\R$ is called \emph{Higson} if for every entourage $E\s X^2$ the map 
\f{
d\varphi|_E:E &\to \R\\
(x,y)&\mapsto \varphi(x)-\varphi(y)
}
vanishes at infinity. Then the compactification $hX$ of $X$ generated by 
the Higson functions $C_h(X)$ is called the \emph{Higson compactification}. The boundary 
of this compactification $\nu(X)=hX\setminus X$ is called the \emph{Higson 
corona}. 
\end{defn}

The large-scale proximity relation induced on $X$ is the close relation.
\begin{rem}
If $X$ is a proper metric space then $\partial_\close(X)=\nu(X)$.
\end{rem}
\begin{proof} This follows from  by 
\cite[Theorem~20]{Hartmann2019c}. 

We provide an alternative proof: Let $A,B\s X$ be subsets. If $A\close B$ then by \cite[Theorem~5.14]{Naimpally1970} there exists a $\close$-ultrafilter $\sheaff$ on $X$ with $A,B\in \sheaff$. Thus $\sheaff\in\closedop A\cap \closedop B$ is not empty. If on the other hand $A\notclose B$ then $\sheaff\in \closedop A$ implies $\sheaff\not\in\closedop B$. Thus $\closedop A\cap \closedop B=\emptyset$ is empty. This way we have shown that $\close$ is the unique relation on subsets of $X$ that tells when the closure of two subsets meet on the boundary of the Higson compactification. 
\end{proof}

\begin{prop}
      If $X$ is a one-ended proper metric space then $\nu(X)$ is connected. This implies that every coarse compactification of $X$ is connected.
\end{prop}
\begin{proof}
Recall that a metric space has at most one end if for every $A\s X$ we have $A\close A^c$ or one of $A,A^c$ is bounded. Suppose $\pi\s \nu(X)$ is a clopen subset. Then $\pi\notclose \pi^c$. Then there exist $A,B\s X$ with $\pi\s \closedop A,\pi^c\s\closedop B$ and $A\notclose B$. By the proof of Theorem~\ref{thm:freudenthaluniversal} the inclusion $A\cup B\to X$ is coarsely surjective. Thus one of $A,B$ is bounded which implies one of $\pi,\pi^c$ is the empty set.
\end{proof}

Now we select Higson functions which separate coarsely disjoint subsets of $X$. A close examination shows they together with the constant functions already generate the Higson functions.

Let $X$ be a proper metric space. For every two subsets $A,B\s X$ with 
$A\notclose B$ we define
\f{
\varphi_{A,B}:X&\to \R\\
x&\mapsto \frac{d(x,A)}{d(x,A)+d(x,B)}
}
where we assume without loss of generality $d(A,B)>0$. If $F\s C^*(X)$ is a subset 
$\mathcal A(F)$ denotes the intersection of all algebras in $C^*(X)$ which 
contain $F$. 

\begin{prop}
There is an isomorphism of $C^*$-algebras 
\[
\overline{\mathcal A((\varphi_{A,B})_{A\notclose B}\cup\{1\})}=C_h(X).  
\]
Here the closure is in $C^*(X)$ with the $\sup$-metric. 
\end{prop}
\begin{proof}
Suppose $A,B\s X$ are subsets with $A\notclose B$. By 
\cite[Lemma~2.2]{Dranishnikov1998} the function $\varphi_{A,B}$ is Higson. 
Thus we have shown $(\varphi_{A,B})_{A\notclose B}\s C_h(X)$.

Now we show $(\bar\varphi_{A,B})_{A\notclose B}$ separates points of $\nu(X)$: 
Let $\sheaff,\sheafg\in \nu(X)$ be points with $\sheaff\bar\lambda_\close 
\sheafg$. Then there exist $A\in \sheaff, B\in \sheafg$ with $A\notclose B$. 
Then
\f{
\bar\varphi_{A,B}([\sheaff])
& =\sheaff\mhyphen\lim \varphi_{A,B}\\
&=0\\
&\not=1\\
&=\sheafg\mhyphen\lim \varphi_{A,B}\\
&=\bar\varphi_{A,B}([\sheafg]).
}
Thus $\bar \varphi_{A,B}$ separates $\sheaff,\sheafg$.

By \cite[Theorem~2.1]{Ball1982} the $(\varphi_{A,B})_{A\notclose B}$ generate 
the compactification $hX$. Now we use \cite[Theorem~3.4]{Ball1982} and obtain 
the result.
\end{proof}

Call a metric space $W$ with $\{x,y\in W:d(x,y)\le R,x\not=y\}$ finite for every $R>0$ \emph{discrete coarse} \cite[Example 2.7]{Roe2003}.

\begin{rem}
If $X$ is an unbounded proper metric space then it contains a sequence $(x_i)_i\s X$ with $d(x_i,x_j)>i$ for $j<i$. Thus $(x_i)_i$ is discrete coarse. It is easy to check every bounded function on $(x_i)_i$ is Higson. Then $h((x_i)_i)=\beta(\N)$ and $\nu((x_i)_i)=\beta(\N)\setminus \N$. Since $h,\nu$ preserves monomorphisms $h((x_i)_i),\nu((x_i)_i)$ arise as subspaces of $h(X),\nu(X)$. Thus $\nu(X)$ contains a copy of $\beta(\N)\setminus \N$ and $hX$ contains a copy of $\beta(\N)$. This fact has already been proved in \cite[Theorem 3]{Keesling1994}.
\end{rem}

\begin{prop}
       If $X$ is a proper metric space then the union of $\closedop W$ over every 
    discrete coarse subspace $W$ of $X$ is dense in $\nu(X)$.
\end{prop}
\begin{proof}
      Define 
      \f{
      \Phi:\bigsqcup_{W\s X \mbox{ discrete}} \nu(W)&\to \nu(X)\\
      \sheaff&\mapsto i_*\sheaff
      }
      where $i: W\to X$ is the inclusion. We show
      \[
            \Phi^*:C(\nu(X))\to C(\bigsqcup_{W\s X\mbox{ discrete}}\nu(W))
      \]
 is injective. Note $C(\nu(X))=C_h(X)/C_0(X)$ and 
 \f{
 C(\bigsqcup_W\nu(W))
 &=\prod_W C(W)\\
 &=\prod_W C_h(W)/C_0(W).
 }
Let $\varphi\in C_h(X)$ be a Higson function. We need to show if $(\varphi\circ 
\Phi)_W\in C_0(W)$ for every discrete subset $W\s X$ then $\varphi\in C_0(X)$. 
Assume for contradiction that $\varphi$ does not converge to zero at infinity. Then there exists $\varepsilon>0$ such that for every $i\in \N$ there is some $x_i\not\in B(x_0,i)$ (Here $x_0\in X$ is a fixed point and $B(x_0,i)$ denotes the ball of radius $i$ around $x_0$) with the property $|\varphi(x_i)|\ge \varepsilon$. Now choose a subsequence $(x_{i_k})_k$ with $\{x_{i_k}:k\}$ discrete. Then $\varphi|_{\{x_{i_k}:k\}}\not\in C_0(\{x_{i_k}:k\})$. Since bounded functions on $\nu(X)$ separate points from closed sets we have shown that the closure of $\im \Phi$ is $\nu(X)$. 

The closure of $\im \Phi$ is $\nu(X)$ since $\bigcup_W \closedop W\s \closedop A$ implies the inclusion $i:A\to X$ is coarsely surjective (Every unbounded subset of $X$ contains a discrete subset). 

Suppose $X$ is $R$-discrete for some $R>0$. If $X$ is not discrete there always exists an ultrafilter on $X$ which does not contain a discrete subspace. Define a filter 
\[
\sheaff=\{X\setminus W: W\s X \mbox{ discrete or finite}\}
\]
Then $\sheaff$ is a filter:
\begin{enumerate}
      \item If $X\setminus W,X\setminus V\in \sheaff$ then $W,V$ are discrete 
or finite. This implies $W\cup V$ is discrete or finite, thus $(X\setminus 
W)\cap (X\setminus V)=X\setminus(V\cup W)\in \sheaff$.
 \item If $X\setminus W\in \sheaff$ and $X\setminus W\s X\setminus V$ then $W$ 
is discrete or finite and $V\s W$. This implies $V$ is discrete or finite. Thus 
$X\setminus V\in \sheaff$.
\end{enumerate}
If $X$ is not discrete or finite then $\sheaff$ is a proper filter. Then there exists an ultrafilter finer than $\sheaff$, it does not contain a discrete subspace.
\end{proof}

%% file: Freudenthal.tex
\section{Space of ends}

The space of ends $\Omega(X)$ of a topological space is the boundary of the Freudenthal compactification $\varepsilon(X)$. In this chapter we will study a coarse version of the Freudenthal compactification which coincides with the topological version of the Freudenthal compactification for a large class of proper metric spaces.

Recall \cite[Problem~41B]{Willard1970}:

\begin{defn}\name{Freudenthal compactification, topological version}
Let $X$ be a rim-compact Tychonoff space. Define a relation $\delta$ on subsets of $X$ by $A\bar\delta B$ for $A,B\s X$ if there is a compact subset $K\s X$ such that $X\setminus K=G\cup H$ is a disjoint union of two open subsets with $\bar A\s G,\bar B\s H$.

The Smirnov compactification of the proximity space $(X,\delta)$ is called the \emph{Freudenthal compactification}. 
\end{defn}
Its boundary is zero dimensional.

Let $A,B\s X$ be subsets of a metric space. Define $A\notclose_f B$ if there 
exist $A'\z A,B'\z B$ with $A'\cup B'=X$ and $A'\notclose B'$.

\begin{prop}
\label{prop:freudenthalequivcoarsetop}
Let $X$ be a proper geodesic metric space. Then it is rim-compact Tychonoff. If $A,B\s X$ are two subsets then $A\delta B$ if and only if $\bar A\cap \bar B\not=\emptyset$ or $A\close_f B$.
\end{prop}
\begin{proof}
Since $X$ is a metric space every point $x\in X$ has a basis of open neighborhoods $\{\mathring B(x,\varepsilon):\varepsilon>0\}$, here $\mathring B(x,\varepsilon)$ denotes the open ball of radius $\varepsilon$ around $x$. Since $X$ is proper the the set 
\[
\overline{B(x,\varepsilon)}\setminus \mathring B(x,\varepsilon)\s \overline{\mathring B(x,\varepsilon)}
\]
is compact. Thus $X$ is rim-compact. Note every metric space is Tychonoff.

Suppose $A,B\s X$ are two subsets with $A\bar \delta B$. Then there exists a compact set $K\s X$ such that $X\setminus K=G\cup H$ with appropriate properties. Let $R>0$ be a number. If $g\in G, h\in H$ are points with $d(g,h)\le R$ then there exists $k\in K$ with 
\f{
d(g,k)+d(k,h)
&=d(g,h)\\
&\le R.
}
Now $K$ is bounded thus there exists $S\ge 0,x_0\in X$ with $K\s B(x_0,S)$. Then $g,h\in B(x_0,S+R)$. This proves $G\notclose H$. Thus $A\notclose_f B$. Since $\delta$ is compatible with the topology on $X$ the relation $\bar A \cap \bar B=\emptyset$ follows.

Suppose $A,B\s X$ are two subsets with $A\notclose_f B$ and $\bar A\cap \bar B\not= \emptyset$. The first relation implies there are $A'\z A,B'\z B$ with $X=A'\cup B',A'\notclose B'$. Then there exists a bounded set $K'\s X$ such that $d(A'\setminus K',B'\setminus K')>1$. Define $A''=\bigcup_{a\in A'\setminus K'}\mathring B(a,1/4)$ and $B''=\bigcup_{b\in B'\setminus K'}\mathring B(b,1/4)$. Now since $X$ is normal there exist open sets $U\z \bar A, V\z \bar B$ with $U\cap V=\emptyset$. The set 
\[
K:=A''^c\cap U^c\cap B''^c\cap V^c\s (A'\setminus K')^c\cap(B'\setminus K')^c=K'
\]
is bounded and closed. Since $X$ is proper this set is compact. We define $G=A''\cup U,H=B''\cup V$. Then $G,H$ are open and disjoint. We have 
\f{
X\setminus K
&=A''\cup U\cup B'' \cup V\\
&=G\cup H
}
and $A\s G,B\s H$. Thus we have shown $A\bar \delta B$. 
\end{proof}

 It is easy to see that $\close_f$ is a large-scale proximity relation. Thus $\bar X^{\close_f}$ is a coarse compactification of $X$. By Proposition~\ref{prop:freudenthalequivcoarsetop} the space is homeomorphic to $\varepsilon(X)$ if $X$ is proper geodesic metric. By slight abuse of notation we write $\Omega(X),\varepsilon(X)$ for the coarse versions of the space of ends, Freudenthal compactification as well.

\begin{defn}
Let $X$ be a metric space with basepoint $x_0\in X$. A bounded continuous map 
$\varphi:X\to \R$ is called \emph{Freudenthal} if for every $R\ge 0$ there exists $K\ge 0$ such that $d(x,y)\le R,d(x_0,x)\le K, d(x_0,y)\le K$ implies $\varphi(x)=\varphi(y)$. We write $C_f(X)$ for the ring of Freudenthal functions on $X$.
\end{defn}

\begin{lem}
Let $X$ be a proper metric space. A bounded continuous function $\varphi:X\to \R$ is Freudenthal if and only if it can be extended to $\bar X^f$.
\end{lem}
\begin{proof}
Without loss of generality assume $X$ is $R$-discrete for some $R>0$. 

Suppose a bounded continuous function $\varphi:X\to \R$ is Freudenthal. If $\sheaff$ is an ultrafilter on $X$ define $\bar\varphi(\sheaff)=\sheaff\mhyphen\lim \varphi$. We show $\varphi$ is well defined: Let $\sheaff,\sheafg$ be ultrafilters on $X$ with $\sheaff\mhyphen\lim\varphi\not=\sheafg\mhyphen\lim\varphi$. Then $X=\iip \varphi {(-\infty,\frac{\sheaff\mhyphen\lim \varphi+\sheafg\mhyphen\lim\varphi}2)}\cup \iip \varphi {[\frac{\sheaff\mhyphen\lim \varphi+\sheafg\mhyphen\lim\varphi}2,\infty)}$ and $\iip \varphi {(-\infty,\frac{\sheaff\mhyphen\lim \varphi+\sheafg\mhyphen\lim\varphi}2)}\notclose \iip \varphi {[\frac{\sheaff\mhyphen\lim \varphi+\sheafg\mhyphen\lim\varphi}2,\infty)}$. Thus $\sheaff\bar\lambda_f \sheafg$.

We show $\bar \varphi$ is continuous: Choose an Interval $I\s \R$ such that $\im\varphi\s I$ and consider $\varphi$ as a map $X\to I$. Let $S,T\s I$ be subsets such that $\bar S\cap \bar T=\emptyset$. Then there is some subset $C\s I$ with $S\s C,T\s C^c$ and $\bar C\cap \bar T=\emptyset,\overline{ C^c}\cap S=\emptyset$. Then we obtain $\iip\varphi C\z \iip \varphi S,\iip \varphi{C^c}\z \iip \varphi T$ and $X=\iip \varphi C\cup \iip \varphi{C^c}$. Now let $R\ge 0$ be a number then there exists a bounded set $B\s X$ such that $d(x,y)\le R,\varphi(x)\not=\varphi(y)$ implies $x,y\in B$. Thus if $x\in \iip \varphi C,y\in \iip \varphi {C^c}$ and $d(x,y)\le R$ then $x,y\in B$. This implies $\iip\varphi C \notclose \iip \varphi {C^c}$. Thus $\iip \varphi S \notclose_f \iip \varphi T$. This shows 
\f{
\iip {\bar \varphi} S\cap \iip {\bar \varphi} T
&=(\iip \varphi S\cup\closedop {\iip\varphi S})
\cap (\iip \varphi T\cup\closedop {\iip\varphi T})\\
&=\emptyset.
}
Thus $\bar\varphi$ is continuous. 

Now we show $C_f(X)$ separates points of $\partial_f(X)=\bar X^f\setminus X$. If $\sheaff,\sheafg$ are ultrafilters on $X$ with $\sheaff\bar\lambda_f\sheafg$ then there are $A\in \sheaff,B\in \sheafg$ with $A\notclose_f B$. Then there exists $C\s X$ with $A\s C,B\s C^c$ and $C\notclose C^c$. Define
\begin{align*}
\varphi:X&\to \R\\
x&\mapsto\begin{cases}
1 & x\in C\\
0 & x\in C^c.
\end{cases}
\end{align*}
Then $\varphi$ is a Freudenthal function. Now the extension $\bar \varphi$ of $\varphi$ separates $\sheaff$ from $\sheafg$. Then by \cite{Ball1982} the ring $C_f(X)$ determines the compactification $\bar X^f$ of $X$.
\end{proof}

\begin{thm}
\label{thm:freudenthaluniversal}
Let $X$ be a proper metric space. The boundary of the Freudenthal 
compactification $\Omega X=\varepsilon X\ohne X$ of $X$ is totally disconnected. 
If $(\bar X,X)$ is another coarse compactification whose boundary is 
totally disconnected then it factors through $\varepsilon X$. The association 
$\Omega$ is a functor that maps coarse maps modulo close to continuous maps.
\end{thm}
\begin{rem}
Compare this result with \cite[Theorem 1]{Peschke1990}. The Freudenthal 
compactification of a topological space with nice properties is universal among 
compactifications with totally disconnected boundary.
\end{rem}
\begin{proof}
At first we show $\Omega X$ is totally disconnected. It is sufficient to show 
that there exists a basis consisting of clopen subsets in 
$\partial_{\close_f}(X)$. If $A\s X$ has the property $A\notclose A^c$ then 
$\closedop A=\closedop{A^c}^c$ is both open and closed. Now we show $(\closedop 
A)_{A\notclose A^c}$ are a basis for the topology on $\Omega X$. Note already 
$(\closedop A^c)_{A\s X}$ are a base for a topology on $\Omega X$. Let $A\s X$ 
be 
a subset and $\sheaff\in \closedop A^c$ be a $\close_f$-ultrafilter. Then 
there exists $B\in \sheaff$ with $B\notclose_f A$. Thus there exists $A'\z 
A,B'\z B$ with $A'\cup B'=X$ and $A'\notclose B'$. This implies $A'\notclose_f 
B$ thus $\sheaff\in \closedop {A'}^c\s\closedop A^c$. Now $A'$ is of the type 
$A'\notclose{A'}^c$.

Suppose $r$ is a close relation on $X$ such that $\partial_r X$ is totally 
disconnected. Then there exists a basis of clopen sets on $\partial_r X$. Let 
$\pi\s \partial_r X$ be a clopen subset. Thus $\pi\bar r \pi^c$. This implies 
there exist $A,B\s X$ with $A\bar r B$ and $\pi\s \closedop A,\pi^c\s \closedop 
B$. In particular $A\notclose B$ and
\f{
\closedop{A\cup B}
&=\closedop A \cup \closedop B\\
&\z \pi \cup \pi^c\\
&=\partial_r X.
}
This implies the inclusion $A\cup B\to X$ is coarsely surjective. Thus 
$A\notclose_f B$ which implies $\pi\notclose_f \pi^c$. Thus the unique map 
$\varepsilon X \to \partial_r X$ extending the identity on $X$ is well-defined 
and continuous. 

Now we show $\Omega$ is a functor. Let $\varphi:X\to Y$ be a coarse map between 
metric spaces. It is sufficient to show that $A\close_f B$ implies 
$\varphi(A)\close_f \varphi(B)$. Suppose $\varphi(A)\notclose_f \varphi(B)$. 
Then there exist $A'\z \varphi(A),B'\z \varphi(B)$ with $A'\cup B'=Y$ and 
$A'\notclose B'$. This implies $\iip \varphi{A'}\z A,\iip \varphi{B'}\z 
B,\iip\varphi{A'}\cup\iip \varphi{B'}=X$ and $\iip \varphi {A'}\notclose \iip 
\varphi{B'}$. Thus $A\notclose_f B$.
\end{proof}

\begin{cor}\name{Protasov}
      Let $X$ be a proper metric space. If $\asdim(X)=0$ then $\nu(X)$ and 
$\Omega(X)$ coincide.
\end{cor}
\begin{rem}
      Compare this result with \cite[Lemma~4.3]{Protasov2003}. We prove the 
same result using universal properties.
\end{rem}
\begin{proof}
     Since $\asdim(X)=0$ the space $\nu(X)$ is zero dimensional by 
\cite{Dranishnikov1998}, \cite{Dranishnikov2000}. This implies $\nu(X)$ is 
totally disconnected. By Theorem~\ref{thm:freudenthaluniversal} there exists a 
unique surjective map $h(X)\to \varepsilon(X)$ which extends the identity on 
$X$. Now by Remark~\ref{rem:uniquequotient} there exists a 
unique surjective map $h(X)\to \varepsilon(X)$. Since the composition of both 
maps $h(X)\to h(X)$ and $\varepsilon(X)\to \varepsilon(X)$ are unique 
surjective they agree with the identity. This proves the spaces 
$\Omega(X),\nu(X)$ are homeomorphic.  
\end{proof}

%% file: Gromov.tex
\section{Gromov boundary}

The Gromov boundary is the last interesting example in this paper. There is a quotient map from the Higson compactification to the Gromov compactification since it is a coarse compactification. We are going to present in this chapter maps in the other direction.

If $X$ is a metric space and $x_0\in X$ a fixed point then the \emph{Gromov product} of two points $x,y\in X$ is defined as
\[
(x|y):=1/2(d(x,x_0)+d(y,x_0)-d(x,y)).
\]

\begin{defn}\name{Gromov boundary}
 Let $X$ be a proper geodesic hyperbolic metric space.  A continuous function $\varphi:X\to \R$ is called \emph{Gromov} if for every $\varepsilon>0$ there exists $K>0$ such that
\[
      (x|y)>K \,\to\, |\varphi(x)-\varphi(y)|<\varepsilon.
\]
The Gromov functions determine a compactification of $X$ called the \emph{Gromov 
compactification} $gX$. The boundary $\partial X=gX\setminus X$ is called the 
\emph{Gromov boundary}.
\end{defn}

\begin{rem}
\label{rem:gromovboundary}
Let $X$ be a proper hyperbolic geodesic metric space. Two sequences $(a_i)_i,(b_i)_i\s X$ converge to the same point on the Gromov boundary $\partial X$ if and only if 
\[
      \liminf_{i,j\to \infty}(a_i|b_j)=\infty.
\]
If $p\in \partial X$ define
\[
U_1(p,r)=\{q\in \partial(X): [(x_n)_n]=p,[(y_n)_n]=q,\liminf_{i,j\to \infty}(x_i|y_j)\ge r\}
\]
and
\[
U_2(p,r)=\{y\in X: [(x_n)_n]=p,\liminf_{i,j\to \infty}(x_i|y)\ge r\}.
\]
Then $\{U_1(p,r)\cup U_2(p,r):r\ge 0\}$ is a neighborhood basis of $p$ in $gX$.
\end{rem}
\begin{proof}
The first part is \cite[Proposition~4.3]{Fukaya2018}. The second part is \cite[Definition~2.13]{Benakli2002}.
\end{proof}

\begin{ex}\name{Gromov boundary}
Let $A,B\s X$ be subsets of a hyperbolic proper metric space. Define $A\close_g B$ if there are sequences $(a_i)_i\s A,(b_i)_i\s B$ such that 
\[
      \liminf_{i,j\to \infty}(a_i|b_j)=\infty.
\]
 If $A\close B$ then there exist unbounded sequences $(a_i)_i\s A,(b_i)_i\s B$ 
and some $R\ge 0$ such that $d(a_i,b_i)\le R$ for every $i$. This implies 
$\liminf_{i,j\to \infty}(a_i|b_j)=\infty$ thus $A\close_g B$. By \cite[Proposition~9.8]{Grzegrzolka2018a} the relation $\close_g$ is a coarse proximity relation, Thus axioms 1,2,4,5 of Definition~\ref{defn:closerelation} hold. Axiom 3 of Definition~\ref{defn:closerelation} holds trivially, thus $\close_g$ is a large-scale proximity relation.  
\end{ex}

\begin{ex}\name{Gromov boundary}
Let $X$ be a hyperbolic geodesic proper metric space. By Remark~\ref{rem:gromovboundary} we obtain $\partial_{\close_g}=\partial(X)$. 
Here the right side denotes the Gromov boundary of $X$.
\end{ex}

\begin{rem}
\label{rem:qgr}
      If $X$ is a hyperbolic metric space and $\gamma,\delta:\Z_+\to X$ are 
quasigeodesic rays in $X$ then $\gamma(\Z_+)\close \delta(\Z_+)$ implies 
there exists some entourage $E\s X\times X$ with $E[\gamma(\Z_+)]\z 
\delta(\Z_+)$ and $E[\delta(\Z_+)]\z \gamma(\Z_+)$.
\end{rem}
\begin{proof}
      By \cite[Definition~6.16]{Roe2003} a map $\gamma:\Z_+\to X$ is a 
quasigeodesic ray if there are constants $R>0,S\ge 0$ with
\[
      \ii R|i-j|-S\le d(\gamma(i),\gamma(j))\le R|i-j|+S
\]
for every $i,j\in\Z_+$. It follows from \cite[Theorem~6.17]{Roe2003} that there 
exists some $T\ge 0$ such that $d(\gamma(\Z_+),\delta(\Z_+))\le T$.
\end{proof}

\begin{rem}
\label{rem:monohigson}
     Let $X$ be a geodesic metric space and $\tilde\gamma:\R_+\to X$ a 
geodesic ray. If $\gamma:\Z_+\to X$ is close to $\tilde\gamma$ then it is 
coarsely injective coarse and the induced map $\nu(\gamma):\nu(\Z_+)\to \nu(X)$ 
is a closed embedding.
\end{rem}
\begin{proof}
      This is \cite[Lemma~39]{Hartmann2019a}.
\end{proof}

\begin{thm}
\label{thm:closedembedding}
      Let $X$ be a hyperbolic geodesic proper metric space. Then there is      
 a closed embedding $\Phi: \nu (\Z_+)\times \partial X\to \nu (X)$.
\end{thm}
\begin{rem}
      Compare this result with \cite[Theorem~10.1.2]{Buyalo2007} which states 
that for every proper geodesic hyperbolic metric space the inequality
\[
      \asdim(X)\ge\dim \partial(X)+1
\]
holds. Note $\asdim(\Z_+)=1$ and $\asdim(X)=\dim(\nu(X))$ for every proper 
metric space \cite{Dranishnikov1998},\cite{Dranishnikov2000}. By 
\cite[Theorem~3]{Morita1977} we obtain $\dim(\nu(\Z_+)\times \partial X)=\dim 
(\partial X)+1$. Thus we obtain a new proof for the above inequality.
\end{rem}

\begin{proof}
     Let $S\s X$ be an $R$-discrete for some $R>0$ subspace such that the 
inclusion $S\to X$ is coarsely surjective. Let $p\in \partial(X)$ be a point. 
Then $p$ is represented by an isometry $\tilde\gamma:\Z_+\to X$. Now choose 
$\gamma:\Z_+\to S$ close to $\tilde \gamma$. Then $\gamma$ is a quasigeodesic 
ray in $S$. If $\sheaff$ is a nonprincipal 
ultrafilter on $\Z_+$ then $\gamma_*\sheaff$ is a nonprincipal ultrafilter on 
$S$. Define 
 \f{
     \Phi: \nu (\Z_+)\times \partial X&\to \nu (X)\\
      ([\sheaff],p)&\mapsto [\gamma_*\sheaff]
      }
Since $\gamma$ is a coarse map $\sheaff\lambda_\close \sheafg$ 
implies $\gamma_*\sheaff\lambda_\close \gamma_*\sheafg$ for every 
$\sheaff,\sheafg\in \nu(\Z_+)$. If $\gamma,\delta:\Z_+\to S$ represent the same 
point in $\partial X$ then there exists $K\ge 0$ such that 
$d(\gamma(n),\delta(n))\le K$. This implies 
$\gamma_*\sheaff\lambda_\close\delta_*\sheaff$ for every $\sheaff\in 
\nu(\Z_+)$. Thus $\Phi$ is a well defined map.

Now we show $\Phi$ is injective: Let $\sheaff,\sheafg\in 
\nu(\Z_+),\gamma,\delta\in\partial X$ be elements with 
$\gamma_*\sheaff\lambda_\close\delta_*\sheafg$. Then $\gamma(\Z_+)\close 
\delta(\Z_+)$. Then Remark~\ref{rem:qgr} implies that $\gamma,\delta$ represent 
the same element. Without loss of generality assume that $\gamma=\delta$. Then 
$\gamma_*\sheaff \lambda_\close \gamma_*\sheafg$ implies $\sheaff\lambda_\close 
\sheafg $ by Remark~\ref{rem:monohigson}.

Now we show $\Phi$ is open: Denote by $p_2:\nu(\Z_+)\times \partial(X)\to 
\partial (X)$ the projection to the second factor. Then $p_2\circ\ii \Phi$   
equals the quotient map $q_X$ of Theorem~\ref{prop:quotientclosefiner} 
restricted to the image of $\Phi$. Let $U\s \partial X$ be open then
\f{
\Phi(U\times \nu(\Z_+))
&=(p_2\circ \ii \Phi)^{-1}(U)\\
&=\iip {q_X} U
}
is open. If $V\s \nu(\Z_+)$ is open then $\Phi(\partial X\times 
V)=\bigcup_\gamma \gamma_*V$ is a union of open sets and thus open.

Thus we have shown $\ii\Phi:\Phi(\nu(\Z_+)\times \partial(X))\to 
\nu(\Z_+)\times \partial X$ is bijective and continuous. Since 
$\Phi(\nu(\Z_+)\times \partial(X))$ is compact and $\nu(\Z_+)\times 
\partial(X)$ 
Hausdorff the map $\Phi$ is a homeomorphism onto its image.
\end{proof}

\begin{prop}
\label{prop:treeretract}
      If $T$ is a tree then the space $\Phi(\partial(T)\times \nu(\R_+))$ is a 
retract of $\nu(T)$.
\end{prop}
\begin{proof}
We first show that $\Phi(\partial(T)\times \nu(\R_+))$ is a retract of 
$\varpi:=\bigcup_{A\s T,|\partial (A)|=1}\closedop A\s \nu(T)$:

      If $\sheaff\in \closedop A$ with $|\partial (A)|=1$ then there is a 
geodesic ray $\gamma$ on $T$ such that $\sheaff$ converges to the point 
represented by $\gamma$ in the Gromov compactification.  Since $T$ is CAT(0) 
and $\gamma(\R_+)$ is convex and complete 
\cite[Proposition~II.2.4]{Bridson1999} provides us with a projection map 
$\pi:T\to \gamma(\R_+)$ such that $d(x,\pi(x))=d(x,\gamma(\R_+))$ for every 
$x\in T$.

Then $\{\pi(A):A\in \sheaff,\partial(A)=[\gamma]\}$ define a base for a 
$\close$-ultrafilter $\sheaff_1$ on $T$:
\begin{enumerate}
      \item If $A,B\in \sheaff$ then $A\close B$. Thus there exist unbounded 
sequences $(a_i)_i\s A,(b_i)_i\s B$ with $d(a_i,b_i)\le R$. Then $\pi(a_i)_i\s 
\pi(A),\pi(b_i)_i\s \pi(B)$. By \cite[Proposition~II.2.4.4)]{Bridson1999} the 
map $\pi$ does not increase distances. Thus $d(\pi(a_i),\pi(b_i))\le R$ for 
every $i$. Since $\liminf_{i,j\to\infty}(a_i|\gamma(j))=\infty$ the sequence 
$(\pi(a_i))_i$ is not bounded. This way we have shown that $\pi(A)\close \pi(B)$. 
\item If $A\cup B\in \sheaff_1$ then $A\cup B=\pi(C)$ for some $C\in \sheaff$. 
Define $A'=\{a\in C:\pi(a)\in A\}$ and $B'=\{b\in C:\pi(b)\in B\}$. Then 
$\pi(A')=A,\pi(B')=B$ and $A'\cup B'=C$. Now $A'\in \sheaff$ or $B'\in 
\sheaff$ which implies $A\in \sheaff_1$ or $B\in\sheaff_1$.
\end{enumerate}
Now define a map 
\f{
r:\varpi&\to \nu(T)\\
\sheaff&\mapsto \sheaff_1.
}
Let $\sheaff,\sheafg\in \nu(T)$ be two elements with $\sheaff\lambda_\close 
\sheafg$. If $A\in \sheaff_1,B\in\sheafg_1$ then there are 
$A'\in\sheaff,B'\in\sheafg$ with $\pi(A')=A,\pi(B')=B$. Now $A'\close B'$ 
implies $\pi(A')\close \pi(B')$ as above. This implies 
$\sheaff_1\lambda_\close\sheafg_1$.

Now we show $r$ is continuous on $\varpi$: Suppose $A,B\s T$ are subsets with $(\closedop A\cap\varpi)\cap (\closedop B \cap \varpi)\not=\emptyset$. Then
\f{
r(\closedop A\cap \varpi)\cap r(\closedop B\cap \varpi)
&=(\bigcup_{A'\s A,|\partial(A)|=1}\closedop{\pi(A')})\cap (\bigcup_{A'\s A,|\partial(A)|=1}\closedop{\pi(A')})\\
&=\bigcup_{A'\s A,|\partial(A')|=1,B'\s B,|\partial(B')|=1} (\closedop {A'}\cap\closedop{B'})\\
&{\not=}\emptyset.
}
To see the last inequality choose $A':=(a_i)_i\s A,B':=(b_i)_i\s B$ unbounded with $d(a_i,b_i)\le R$ for every $i$ and some $R\ge 0$. If necessary we can choose a subsequence of $(a_i)_i$ such that $(a_i)_i$ converges to a point on the Gromov boundary. Then $(b_i)_i$ converges to the same point. Thus we can assume $|\partial(A')|=1=|\partial(B')|$. Then $\pi(A')\close \pi(B')$, in fact both sets are finite Hausdorff distance apart. Thus $\closedop{\pi(A')}\cap \closedop{\pi(B')}=\not\emptyset$. We just showed $r$ is uniformly continuous with regard to the unique uniformity on the compact space $\nu(T)$.

Note that $\bigcup_{A\s T,|\partial(A)|=1}\closedop A$ is dense in $\nu(T)$: 
Consider the closure of $\bigcup_{A\s T,|\partial(A)|=1}\closedop A$. If $B\s 
T$ is a subset then there exists a sequence $(b_i)_i\s B$ such that $(b_i)_i$ 
converges to a point $\gamma\in\partial(X)$ in the Gromov compactification. 
This means $|\partial((b_i)_i)|=1$. Thus $\overline{\bigcup_{A\s 
T,|\partial(A)|=1}\closedop A}=\nu(T)$.

Then \cite[Theorem~8.3.10]{Engelking1989} implies the retract map $r$ can be 
extended to $\nu(T)$.
\end{proof}

\begin{rem}
The results in Theorem~\ref{thm:closedembedding} and Proposition~\ref{prop:treeretract} are functorial: If $\alpha:T\to S$ is a coarse map between trees such that $\alpha\circ \gamma$ is coarsely injective if $\gamma:\R_+\to X$ is coarsely injective coarse then there is a continuous map
\f{
r\circ \nu(\alpha)\circ\Phi:\partial (T)\times \nu(\R_+)&\to \partial (S)\times\nu(\R_+)\\
([\gamma],[\sheaff])&\mapsto ([\alpha\circ \gamma],[\sheaff]).
}
If $\alpha$ is a coarse equivalence then $\partial(\alpha):\partial(T)\to \partial(S)$ is a homeomorphism since the Gromov boundary is a functor on coarse equivalences. This implies $r\circ \nu(\alpha)\circ \Phi$ is an isomorphism in the topological category. 
\end{rem}

%% file: Remarks.tex
\section{Remarks}